\documentclass{IET}

\usepackage{amsmath, amsthm, amscd, amsfonts, amssymb, graphicx, color,  mathrsfs}
\newtheorem{theorem}{Theorem}

\newtheorem{remark}{Remark}
\newtheorem{lemma}{Lemma}

\newtheorem{corollary}{Corollary}
\newtheorem{definition}{Definition}
\DeclareMathAlphabet{\mathpzc}{OT1}{pzc}{m}{it}

\begin{document}

\title{$L^p$-estimates for the Schr\"odinger equation associated to the harmonic oscillator}

\author{Duv\'an Cardona}
\affil{Pontificia Universidad Javeriana, Mathematics Department, Bogot\'a-Colombia}
\affil[1]{cardonaduvan@javeriana.edu.co}

\abstract{In this paper we obtain  some Strichartz estimates for the Schr\"odinger equation associated to the harmonic oscillator and the Laplacian. Our main tool will be some embeddings between Lebesgue spaces and suitable Triebel-Lizorkin spaces. MSC 2010. Primary: 42B35, Secondary: 42C10, 35K15. To appear in Electron. J. Differential Equations. Received: Feb 6-2018; Accepted: Aug 8-2018.}

\maketitle
\section{Introduction}
Let us consider the quantum harmonic oscillator $H:=-\Delta+|x|^2$ on $\mathbb{R}^n$ where $\Delta$ is the standard Laplacian.  In this paper we obtain regularity for the Schr\"odinger equation (associated to $H$)  given by
\begin{equation}\label{SEq}
iu_t(t,x)-Hu(x,t)=0, 
\end{equation}
with initial data $u(0,\cdot)=f.$ As it is well known, this is an important model in quantum mechanics (R. P. Feynman, and A.R. Hibbs, \cite{FeyHib65}). As consequence of such estimates we also provide estimates for the classical Schr\"odinger equation
\begin{equation}\label{SEq'}
iu_t(t,x)+\Delta u(x,t)=0. 
\end{equation}
The regularity for the problem \eqref{SEq} has been extensively developed, some works on the subject are S. Thangavelu \cite[Section 5]{thangavelu0}, B. Bongioanni and J. L. Torrea \cite{BonTorrea}, B. Bongioanni and K. M. Rogers \cite{BonRog} and K. Yajima \cite{Yajima} and references therein. On the other hand, regularity properties for \eqref{SEq'} can be found in the seminal work of J. Ginibre and G. Velo \cite{GinVel}, and in A. Moyua and L. Vega \cite{MoyuaVega}, M. Keel and T. Tao \cite{KeelTao} and references therein. The works L. Carleson \cite{Carleson} and B. Dahlberg and C. Kenig \cite{DalKen} include  pointwise convergence theorems for the solution $u(x,t)=e^{it\Delta}f$. 

It was proved in \cite{MoyuaVega}  the following sharp theorem: for  $\frac{2(n+2)}{n}\leq p\leq \infty,$ and $2\leq q< \infty$ with $\frac{1}{q}\leq \frac{n}{2}(\frac{1}{2}-\frac{1}{p}),$  
\begin{equation}\label{BV}
\Vert u(t,x)\Vert_{L^p_{x}(\mathbb{R}^n\,, L^q_{t}[0,2\pi])}\leq C_s\Vert f\Vert_{\mathcal{H}^s(\mathbb{R}^n)}
\end{equation}
holds true for all $s\geq s_{n,p,q}:=n(\frac{1}{2}-\frac{1}{p})-\frac{2}{q}.$ If $s<s_{n,p,q}$ then \eqref{BV} is false. In the result above $\mathcal{H}^s$ is the Sobolev space associated to $H$ and with norm $\Vert f\Vert_{\mathcal{H}^s}:=\Vert H^{s/2}f \Vert_{L^2}.$ The proof of  \eqref{BV} involves Strichartz estimates of M. Keel and T. Tao \cite{KeelTao} and
 Wainger's Sobolev embedding theorem. It is important to mention that the machinery of the work M. Keel and T. Tao \cite{KeelTao} implies the following estimate
 \begin{equation}\label{KT}
\Vert u(t,x)\Vert_{L^q_{t}([0,2\pi],L^p_{x}(\mathbb{R}^n) )}\leq C_p\Vert f\Vert_{L^2(\mathbb{R}^n)},
\end{equation}
 for $2\leq q<\infty$ and $\frac{1}{q}= \frac{n}{2}(\frac{1}{p}-\frac{1}{2}),$ excluding the case $(p,q,n)=(\infty, 2, 2).$ On the other hand, H. Koch and
D. Tataru have proved the estimate \eqref{KT} for Schr\"odinger type operators in more general contexts, but including the operator $H,$ and they have proved that estimates of this type can be not obtained for $2\leq p<\frac{2n}{n-2}.$

A remarkable formula that links the solution of \eqref{SEq} to that of the classical
Schr\"odinger equation (see P. Sj\"ogren, and J.L. Torrea \cite{SjTo}) is the following
\begin{equation}\label{SjToFormula}
\Vert e^{-it((-\Delta+|x|^2))}f\Vert_{L^q[(0,\frac{\pi}{4}), L_x^p(\mathbb{R}^d)]}=\Vert e^{it\Delta}f\Vert_{L^q[(0,\infty), L_x^p(\mathbb{R}^d)]}
\end{equation}
for $1\leq p,q\leq \infty$ and $\frac{2}{q}=n(\frac{1}{2}-\frac{1}{p}).$ As it was pointed out in \cite{SjTo}, 
the interval of integration in the $t$ variable is now bounded, \eqref{KT} remains true if the
equality in \eqref{SjToFormula} is replaced by the inequality $ n(\frac{1}{2}-\frac{1}{p})\leq\frac{2}{q},$ and the interval $(0,\frac{\pi}{4})$ can be replaced by $(0,\frac{\pi}{2})$ and in a such case the two norms are equivalents for real functions $f$.   In particular,  \eqref{SjToFormula} shows that \eqref{KT} is equivalent to the following Strichartz estimate (see \cite{KochTat})
\begin{equation}\label{KT'}
\Vert e^{it\Delta}f\Vert_{L^q[(0,\infty), L_x^p(\mathbb{R}^d)]}\leq C\Vert f\Vert_{L^2(\mathbb{R}^n)}
\end{equation}
which holds if and only if $n=1$ and $2\leq p\leq \infty,$ $n=2$ and $2\leq p<\infty$ and $2\leq p<\frac{2n}{n-2}$ for $n\geq 3.$ 

The novelty of this paper is that we provide regularity results for the Sch\"odinger equation associated to $H,$ involving $L^p$-Sobolev norms for the initial data instead of the $L^2$ and $L^2$-Sobolev bounds mentioned above.
Our main result in this paper is the following theorem.
\begin{theorem}\label{mainThe}
Let us assume $n>2,$ $2\leq q<\infty$ and $1<p<2$  satisfying $|\frac{1}{2}-\frac{1}{p}|<\frac{1}{2n}.$ Then, the following estimate
\begin{equation}\label{mainestimate}
\Vert u(t,x) \Vert_{L^{p'}_x[\mathbb{R}^n,\,L^q_t[0,2\pi]]}\leq C\Vert f\Vert_{W^{2s,p,H}(\mathbb{R}^n)} 
\end{equation}
holds true for every $s\geq s_q:=\frac{1}{2}-\frac{1}{q}.$ In particular, if $q=2$ we have
\begin{equation}
\Vert u(t,x) \Vert_{L^{p'}_x[\mathbb{R}^n,\,L^2_t[0,2\pi]]}\leq C\Vert f\Vert_{L^{p}(\mathbb{R}^n)}.
\end{equation} Moreover, for $n>2,$  $1<p<2,$ and $1\leq q\leq p',$ we have
\begin{equation}
\Vert u(t,x) \Vert_{L^{p'}_x[\mathbb{R}^n,\,L^q_t[0,2\pi]]}\leq C\Vert f\Vert_{L^{p}(\mathbb{R}^n)},
\end{equation} provided that $|\frac{1}{p}-\frac{1}{2}|<\frac{1}{nq}.$
\end{theorem}
Now, in the following remarks, we briefly discuss some consequences of our main result.
\begin{remark}
The main contribution in  Theorem \ref{mainThe} is the estimate \eqref{mainestimate} which in particular implies an analogue of the Littlewood-Paley theorem (see  \eqref{analoguelittlewood}). Littlewood-Paley  type results can be understood as substitutes of the Plancherel identity on $L^p$-spaces.  
\end{remark}
\begin{remark}
An important consequence of Theorem \ref{mainThe} is the following estimate 
\begin{equation}\label{StrEst}
\Vert e^{it\Delta}f\Vert_{L^q[(0,\infty), L_x^p(\mathbb{R}^d)]}\asymp\Vert u(t,x)\Vert_{L^q_{t}([0,2\pi]\,, L^p_{x}(\mathbb{R}^n))}\leq C \Vert f\Vert_{\mathpzc{F}^s_{p,2}(\mathbb{R}^n)},\,\, s\geq {s_q},
\end{equation} for $2\leq p\leq q<\infty, $ $\frac{2}{q}= n(\frac{1}{2}-\frac{1}{p}),$ (see Theorem \ref{StrEst'}), the inequality
\begin{equation}\label{StrEst''''}
\Vert e^{it\Delta}f\Vert_{L^q[(0,\infty), L_x^{p'}(\mathbb{R}^d)]}\asymp\Vert u(t,x)\Vert_{L^q_{t}([0,2\pi]\,, L^{p'}_{x}(\mathbb{R}^n))}\leq C \Vert f\Vert_{W^{2s,p,H}(\mathbb{R}^n)},\,\, s\geq {s_q},
\end{equation}
for $|\frac{1}{p}-\frac{1}{2}|<\frac{1}{2n},$ $1<p<2,$ $n>2$ and $\frac{2}{q}= n(\frac{1}{p}-\frac{1}{2}),$ (please, let us compare \eqref{StrEst''''} and \eqref{KT}) as well as the estimate
\begin{equation}
\Vert f\Vert_{\mathpzc{F}^{0}_{p,2}(\mathbb{R}^n)}\leq C\Vert e^{it\Delta}f\Vert_{L^q[(0,\infty), L_x^p(\mathbb{R}^n)]}\asymp C\Vert u(t,x)\Vert_{L^q_{t}([0,2\pi]\,, L^p_{x}(\mathbb{R}^n))}
\end{equation} when $2\leq q\leq p<\infty $ provided that $ n(\frac{1}{2}-\frac{1}{p})=\frac{2}{q}.$ In the results above the spaces $\mathpzc{F}^{s}_{p,2}$ are  Triebel-Lizorkin spaces associated to $H$ and they will be introduced in the next section. 
\end{remark}
\begin{remark}
Finally, \eqref{StrEst} links our results with those in \cite{KeelTao} and \cite{SjTo}. For $\frac{1}{q}=\frac{n}{2}(\frac{1}{2}-\frac{1}{p})$ we show in Corollary \ref{improv} that the following estimate
\begin{equation}\label{impro}
\Vert u(t,x)\Vert_{L^p_{x}(\mathbb{R}^n\,, L^q_{t}[0,2\pi])}\leq C_s\Vert f\Vert_{{L}^2(\mathbb{R}^n)}
\end{equation}
holds true provided that $n=1$ and $2\leq p\leq \infty,$ $n=2$ and $2\leq p<\infty$ and $2\leq p<\frac{2n}{n-2}$ for $n\geq 3.$  As a consequence of the embedding $\mathcal{H}^{s}\hookrightarrow {L}^{2}  $ for $s\geq 0$, the estimate \eqref{impro}  improves   \eqref{BV} in any case above.
\end{remark}
This paper is organized as follows. In section  \ref{Preliminaries} we present some basics on the spectral decomposition of the harmonic oscillator and we discuss our analogue of the Littlewood-Paley theorem. Finally, in the last section we provide our regularity results.

\section{Spectral decomposition of the harmonic oscillator and a result of type Littlewood-Paley }\label{Preliminaries}

Let  $H=-\Delta+|x|^2$ be the Hermite operator or (quantum) \emph{harmonic oscillator}. This operator extends to an unbounded self-adjoint operator on $L^{2}(\mathbb{R}^n) $, and its spectrum consists of the discrete set  $\lambda_\nu:=2|\nu|+n,$ $\nu\in \mathbb{N}_0^n,$  with a set of \emph{real eigenfunctions} $\phi_\nu, $ $\nu\in \mathbb{N}_0^n,$ (called Hermite functions) which provide an orthonormal basis of ${L}^2(\mathbb{R}^n).$
 Every Hermite function  $\phi_{\nu}$  on $\mathbb{R}^n$ has the form
\begin{equation}
\phi_\nu=\Pi_{j=1}^n\phi_{\nu_j},\,\,\, \phi_{\nu_j}(x_j)=(2^{\nu_j}\nu_j!\sqrt{\pi})^{-\frac{1}{2}}H_{\nu_j}(x_j)e^{-\frac{1}{2}x_j^2}
\end{equation}
where $x=(x_1,\cdots,x_n)\in\mathbb{R}^n$, $\nu=(\nu_1,\cdots,\nu_n)\in\mathbb{N}^n_0,$ and $$H_{\nu_j}(x_j):=(-1)^{\nu_j}e^{x_j^2}\frac{d^k}{dx_{j}^k}(e^{-x_j^2})$$ denotes the Hermite polynomial of order $\nu_j.$  By the spectral theorem, for every $f\in\mathscr{D}(\mathbb{R}^n)$ we have
\begin{equation}
Hf(x)=\sum_{\nu\in\mathbb{N}^n_0}\lambda_\nu\widehat{f}(\phi_\nu)\phi_\nu(x),\,\,\,
\end{equation} where $\widehat{f}(\phi_v) $ is the Hermite-Fourier transform of $f$ at $\nu$ defined by
\begin{equation} \widehat{f}(\phi_\nu) :=\langle f,\phi_\nu \rangle_{L^2(\mathbb{R}^n)}=\int_{\mathbb{R}^n}f(x)\phi_\nu(x)\,dx.\end{equation}
The main tool in the harmonic analysis of the harmonic oscillator is the Hermite semigroup, which we introduce as follows.  If $P_{\ell},$ $\ell\in2\mathbb{N}_0+n,$ is the projection on $L^{2}(\mathbb{R}^n)$ given by
\begin{equation}
P_{\ell}f(x):=\sum_{2|\nu|+n=\ell}\widehat{f}(\phi_\nu)\phi_\nu(x),
\end{equation}
 then,  the Hermite semigroup (semigroup associated to the harmonic oscillator) $T_{t}:=e^{-tH},$ $t>0$ is given by
 \begin{equation}
 e^{-tH}f(x)=\sum_{\ell}e^{-t\ell}P_{\ell}f(x).
 \end{equation}
 For every $t>0,$ the operator $e^{-tH}$ has Schwartz kernel given by
 \begin{equation}\label{kernelt}
 K_{t}(x,y)=\sum_{\nu\in\mathbb{N}^n_0}e^{-t(2|\nu|+n)}\phi_{\nu}(x)\phi_\nu(y).
 \end{equation}
 In view of Mehler's formula (see Thangavelu \cite{Thangavelu}) the above series can be summed up and we obtain
 \begin{equation}\label{trace2}
 K_{t}(x,y)=(2\pi)^{-\frac{n}{2}}\sinh(2t)^{-\frac{n}{2}}e^{-(\frac{1}{2}|x|^2+|y|^2)\coth(2t)+x\cdot y\cdot \textnormal{csch}(2t))}.
 \end{equation}
 In this paper we want to estimate the mixed norms $L^p_x(L^q_t)$ of  solutions to Sch\"rodinger equations by using the following version of Triebel-Lizorkin associtaed to $H$.
 \begin{definition}
Let us consider $0<p\leq \infty,$ $r\in\mathbb{R}$ and $0<q\leq \infty.$ The Triebel-Lizorkin space associated to $H,$ to the family of projections $P_{\ell},$ $\ell \in 2\mathbb{N}+n,$ and to the parameters $p,q$ and $r$ is defined by those complex functions $f$ satisfying
\begin{equation}
\Vert f\Vert_{\mathpzc{F}^r_{p,q}(\mathbb{R}^n)}:=\left\Vert \left(\sum_\ell \ell^{rq}|P_\ell f|^{q}\right)^{\frac{1}{q}}\right\Vert_{L^p(\mathbb{R}^n)}<\infty.
\end{equation}
 \end{definition}
 The definition considered above differs from those arising with dyadic decompositions (see e.g. \cite{BuiDuong} and \cite{PeXu}). The following are natural embedding properties of such spaces.  $\mathcal{H}^s$ denotes the Sobolev space associated to $H$ and defined by the norm $\Vert f\Vert_{\mathcal{H}^s}:=\Vert H^{s/2}f \Vert_{L^2}.$ Sobolev spaces $W^{2s,p,H}$  in $L^p$-spaces and associated to $H,$ can be defined by the norm $\Vert f\Vert_{W^{2s,p,H}}:=\Vert H^sf\Vert_{L^p}.$
 \begin{itemize}
\item[(1)] ${\mathpzc{F}}^{r+\varepsilon}_{p,q_1}\hookrightarrow {\mathpzc{F}}^{r}_{p,q_1}\hookrightarrow {\mathpzc{F}}^{r}_{p,q_2}\hookrightarrow \mathpzc{F}^{r}_{p,\infty},$  $\varepsilon>0,$ $0<p\leq \infty,$ $0<q_{1}\leq q_2\leq \infty.$
\item[(2)]  $\mathpzc{F}^{r+\varepsilon}_{p,q_1}\hookrightarrow \mathpzc{F}^{r}_{p,q_2}$, $\varepsilon>0,$ $0<p\leq \infty,$ $1\leq q_2<q_1<\infty.$
\item[(3)] $\mathpzc{F}^0_{2,2}=L^2$ and consequently, for every $s\in\mathbb{R},$ $\mathcal{H}^{2s}=\mathpzc{F}^s_{2,2}.$  Other properties associated to Sobolev spaces of the harmonic oscillator can be found in \cite{BuiDuong,BonTorrea} and \cite{PeXu}.
\end{itemize}
Now we discuss a close relation between  $\mathpzc{F}^{0}_{p,2}$ and Lebesgue spaces. If $\psi$ is a smooth function supported in $[\frac{1}{4},2],$ such that $\psi=1$ on $[\frac{1}{2},1],$  
\begin{equation}
\sum_{k=0}^\infty\psi_k(t)=1, \,\, \psi_{k}(t):=\psi(2^{-k}t),
\end{equation} and $A$ is an elliptic pseudo-differential operator on $\mathbb{R}^n$ of order $\nu>0,$ the (dyadic) Triebel-Lizorkin space $F^{r}_{p,q,A}(\mathbb{R}^n)$ associated to $A$  is defined by the norm
\begin{equation}
\Vert f\Vert_{F^{r}_{p,q,A}}:=\Vert \{2^{kr/\nu}\Vert \psi_{k}(A)f \Vert_{L^p}\}\Vert_{\ell^{q}}
\end{equation}
where $r\in\mathbb{R}$ and $0<p,q\leq \infty.$ For $A=H$ or $A=\Delta_{x}$ is known the Littlewood-Paley theorem (see \cite{Duo}) which stands that $F^{0}_{p,2,A}=L^p$ for all $1<p<\infty.$ If $A=\Delta_x,$  one also have
\begin{equation}\label{mLPT}
\left\Vert \left(\sum_{k}|1_{(k,k+1)}(\Delta_x)f|^2\right)^{\frac{1}{2}}\right\Vert_{L^p(\mathbb{R}^n)}\leq C\Vert f\Vert_{L^p},\,2<p<\infty,
\end{equation}
with $C$ depending only on $p.$ However, such inequality is false for $1<p<2.$
For $\ell \in 2\mathbb{N}+n,$ $P_{\ell}=1_{[\ell,\ell+1)}(H)$ and 
\begin{equation}
\Vert f\Vert_{\mathpzc{F}^0_{p,2}}=\left\Vert \left(\sum_{\ell}|1_{[\ell,\ell+1)}(H)f|^2\right)^{\frac{1}{2}}\right\Vert_{L^p(\mathbb{R}^n)}.
\end{equation}
Although in Remark \ref{remark}, we explain in detail that we have not a Littlewood-Paley theorem for $\mathpzc{F}^0_{p,2},$ in the proof of our main theorem we obtain the following estimate for $1<p<2$ (see equation \eqref{LpT})
\begin{equation}\label{analoguelittlewood}
\Vert f\Vert_{\mathpzc{F}^0_{p',2}}=\left\Vert \left(\sum_{\ell}|1_{[\ell,\ell+1)}(H)f|^2\right)^{\frac{1}{2}}\right\Vert_{L^{p'}(\mathbb{R}^n)}\leq C\Vert f\Vert_{L^p}
\end{equation}
provided that $|\frac{1}{p}-\frac{1}{2}|<\frac{1}{2n}.$ Such inequality is indeed, an analogue of \eqref{mLPT}. An immediate consequence is the estimate:
\begin{equation}\label{finalembedding}
\Vert f\Vert_{\mathpzc{F}^s_{p',2}}=\Vert H^sf\Vert_{\mathpzc{F}^0_{p',2}}\leq C \Vert H^sf\Vert_{L^p}=: C\Vert f\Vert_{W^{2s,p,H}}
\end{equation}
provided that $|\frac{1}{p}-\frac{1}{2}|<\frac{1}{2n}.$
\section{Regularity properties}

In order to analyze the mixed norms of solutions of the Sch\"odinger equation we need the following multiplier theorem. The space $L^2_f(\mathbb{R}^n)$ consists of those finite linear combinations of Hermite functions on $\mathbb{R}^n.$
\begin{theorem}
Let us assume that $m\in L^{\infty}(\mathbb{N}_0)$ is a bounded function. Then the multiplier $m(H)$ extends to a bounded operator on $\mathpzc{F}^{0}_{p,q}(\mathbb{R}^n)$ for all  $0<p\leq \infty$ and $0<q\leq \infty.$ Moreover
\begin{equation}
\Vert m(H) \Vert_{\mathscr{B}(\mathpzc{F}^{0}_{p,q})}= \Vert m\Vert_{L^{\infty}}.
\end{equation}
In particular if $m:=1_{[0,\ell']},$  then $S_{\ell'}=1_{[0,\ell']}(H),$ $\Vert S_{\ell'} \Vert_{\mathscr{B}(\mathpzc{F}^{0}_{p,q})}=1$ and 
\begin{equation}\label{approximation}
\lim_{\ell'\rightarrow \infty}\Vert S_{\ell'}f-f \Vert_{\mathpzc{F}^{0}_{p,q}}=0
\end{equation}
uniformly on  the $\mathpzc{F}^{0}_{p,q}$-norm.
\end{theorem}
\begin{proof}
 Let us consider $f\in \mathpzc{F}^{0}_{p,q}.$ Then, $P_\ell(m(H)f)=m(\ell )P_\ell f$ and 
\begin{equation}
\Vert m(H)f \Vert_{\mathpzc{F}^{0}_{p,q}}=\left\Vert \left(\sum_\ell |m(\ell)|^q|P_\ell f|^{q}\right)^{\frac{1}{q}}\right\Vert_{L^p(\mathbb{R}^n)}\leq \sup_{\ell}|m(\ell)|\Vert f \Vert_{\mathpzc{F}^{0}_{p,q}}.
\end{equation}
As consequence
\begin{equation}
\Vert m(H) \Vert_{\mathscr{B}(\mathpzc{F}^{0}_{p,q})}\leq \Vert m\Vert_{L^{\infty}}.
\end{equation} Now, for the reverse inequality, let us choose $f=\phi_{\nu},$ $\ell'=2|\nu|+n.$ Then, $\Vert m(H)f\Vert_{(\mathpzc{F}^{0}_{p,q})}=|m(\ell')|\Vert f \Vert_{(\mathpzc{F}^{0}_{p,q})}$ and as consequence $\Vert m(H) \Vert_{\mathscr{B}(\mathpzc{F}^{0}_{p,q})}\geq \sup_{\ell}|m(\ell)|.$ The second part is consequence of the uniform boundedness principle.
\end{proof}

\begin{remark}\label{remark} As an important consequence of the previous result,  $L^2_f(\mathbb{R}^n)$ is a dense subspace of every space
$\mathpzc{F}^{r}_{p,q},$ in fact, for every $f\in \mathpzc{F}^{r}_{p,q},$ the sequence $\{S_{\ell'}f\}_{\ell'}$ lies in $L^2_f(\mathbb{R}^n)$ and $S_{\ell'}f\rightarrow f$ in norm. For $n=1,$ it is well known that the sequence of operators $\{S_{\ell'}\}_{\ell'}$ is uniformly bounded on $L^p$ if and only if $\frac{4}{3}<p<4,$ so the spaces $\mathpzc{F}^{0}_{p,2}$ does not coincide necessarily with Lebesgue spaces and we have not a general Littlewood-Paley Theorem. Nevertheless, this disadvantaged fact is supered by the efficiency of such spaces when we want to estimate solutions of the Schr\"odinger equation.
\end{remark}
We will use the first part of the remark above in the following result.
\begin{lemma}
Let us consider $f\in \mathpzc{F}^0_{p,2}(\mathbb{R}^n),$ then for all $0<p\leq \infty$  we have 
\begin{equation}\label{eql2lp}
\Vert u(t,x) \Vert_{L^p_x[\mathbb{R}^n,\,L^2_t[0,2\pi]]}=\sqrt{2\pi}\Vert f \Vert_{\mathpzc{F}^0_{p,2}(\mathbb{R}^n)}.
\end{equation}
\end{lemma}
\begin{proof}
In view of \eqref{approximation} we consider by density, $f\in L^2_f(\mathbb{R}^n).$  The solution $u(t,x)$ for \eqref{SEq} is given by
\begin{equation}
u(t,x)=\sum_{\nu\in\mathbb{N}^n_0}e^{-it(2|\nu|+n)}\widehat{f}(\phi_\nu)\phi_\nu(x).
\end{equation}
Then, we have (see \cite{MoyuaVega})
\begin{align*}
\Vert u(t,x) \Vert_{L^2_t[0,2\pi]}^2 =\sum_{\ell}2\pi\,\cdot|P_{\ell}f(x)|^2, 
\end{align*}
in fact, it can be proved by using the orthogonality of trigonometric polynomials. So, we conclude the following fact
\begin{equation}\label{identity1}
\Vert u(t,x) \Vert_{L^2_t[0,2\pi]}=\left(\sum_{\ell}2\pi\,\cdot|P_{\ell}f(x)|^2\right)^{\frac{1}{2}},\,\,\,  f\in L^2_f(\mathbb{R}^n).
\end{equation}
Consequently
\begin{equation}
\Vert u(t,x)\Vert_{L^p_{x}(\mathbb{R}^n,L^2_t[0,2\pi] )}=\sqrt{2\pi}\Vert f\Vert_{\mathpzc{F}^0_{p,2}(\mathbb{R}^n)}.
\end{equation}

\end{proof}

\begin{lemma}\label{T1}
Let $0<p\leq \infty,$ $2\leq q<\infty$ and $s_{q}:=\frac{1}{2}-\frac{1}{q}.$ Then
\begin{equation}
C_p'\Vert f\Vert_{\mathpzc{F}^0_{p,2}}\leq \Vert u(t,x)\Vert_{L^p_{x}(\mathbb{R}^n, L^q_{t}[0,2\pi])} \leq C_{p,s}\Vert f\Vert_{\mathpzc{F}^s_{p,2}}, 
\end{equation}
holds true for every $s\geq s_{q}.$ 
\end{lemma}
\begin{proof}
We consider, by a density argument, $f\in L^2_f(\mathbb{R}^n).$ By following the approach in \cite{BonRog}, in order to estimate the norm  $\Vert u(t,x) \Vert_{L^p_x[\mathbb{R}^n,\,L^q_t[0,2\pi]]}$ we use the Wainger Sobolev embedding Theorem:
\begin{equation}
\left\Vert \sum_{\ell\in \mathbb{Z},\ell\neq 0}|\ell|^{-\alpha}\widehat{F}(\ell)e^{-i\ell t}\right\Vert_{L^q[0,2\pi]}\leq C\Vert F\Vert_{L^r[0,2\pi]},\,\,\alpha:=\frac{1}{r}-\frac{1}{q}.
\end{equation}
For $s_q:=\frac{1}{2}-\frac{1}{q}$  we have
\begin{align*}
\Vert u(t,x)\Vert_{L^q[0,2\pi]} &=\left\Vert \sum_{\nu\in\mathbb{N}^n_0}e^{-it(2|\nu|+n)}\widehat{f}(\phi_\nu)\phi_\nu(x)   \right\Vert_{L^q[0,2\pi]}=\left\Vert \sum_{\ell} e^{-it\ell}P_{\ell}f(x)  \right\Vert_{L^q[0,2\pi]} \\
&\leq C\left\Vert \sum_{\ell} \ell^{s_q}e^{-it\ell}P_{\ell}f(x)  \right\Vert_{L^2[0,2\pi]}\\
&=C\left\Vert \sum_{\ell} e^{-it\ell}P_{\ell}[H^{s_q}f(x)]  \right\Vert_{L^2[0,2\pi]}\\
&=C\left( \sum_{\ell} |P_{\ell}[H^{s_q}f(x)|^2  \right)\\
&:=T'(H^{s_q}f)(x).
\end{align*}
 So, we have
\begin{equation}\label{31}
\Vert u(t,x) \Vert_{L^p_x[\mathbb{R}^n,\,L^q_t[0,2\pi]]}\leq C\Vert T'(H^{s_q}f) \Vert_{L^p(\mathbb{R}^n)}\leq C_{p}\Vert H^{s_q}f\Vert_{\mathpzc{F}^0_{p,2}(\mathbb{R}^n)}=C_{p}\Vert f\Vert_{\mathpzc{F}^{s_q}_{p,2}(\mathbb{R}^n)}.
\end{equation}
We end the proof by taking into account the embedding $\mathpzc{F}^s_{p,2}\hookrightarrow \mathpzc{F}^{s_q}_{p,2}$ for every $s>s_q$ and the following inequality for $2\leq q<\infty$
\begin{equation}
\Vert f\Vert_{\mathpzc{F}^0_{p,2}}=\frac{1}{\sqrt{2\pi}} \Vert T'f\Vert_{L^p}= \frac{1}{\sqrt{2\pi}} \Vert u(t,x) \Vert_{L^p_x[\mathbb{R}^n,\,L^2_t[0,2\pi]]}\lesssim \Vert u(t,x) \Vert_{L^p_x[\mathbb{R}^n,\,L^q_t[0,2\pi]]}.
\end{equation}
\end{proof}

\begin{theorem}
Let us assume $n>2,$ $2\leq q<\infty$ and $1<p<2,$  satisfying $|\frac{1}{2}-\frac{1}{p}|<\frac{1}{2n}.$ Then, the following estimate
\begin{equation}
\Vert u(t,x) \Vert_{L^{p'}_x[\mathbb{R}^n,\,L^q_t[0,2\pi]]}\leq C\Vert f\Vert_{W^{2s,p,H}(\mathbb{R}^n)} 
\end{equation}
holds true for every $s\geq s_q:=\frac{1}{2}-\frac{1}{q}.$ In particular, if $q=2$ we have
\begin{equation}
\Vert u(t,x) \Vert_{L^{p'}_x[\mathbb{R}^n,\,L^2_t[0,2\pi]]}\leq C\Vert f\Vert_{L^{p}(\mathbb{R}^n)}.
\end{equation} Moreover, for $n>2,$  $1<p<2,$ and $1\leq q\leq p',$ we have
\begin{equation}\label{lplq}
\Vert u(t,x) \Vert_{L^{p'}_x[\mathbb{R}^n,\,L^q_t[0,2\pi]]}\leq C\Vert f\Vert_{L^{p}(\mathbb{R}^n)},
\end{equation} provided that $|\frac{1}{p}-\frac{1}{2}|<\frac{1}{nq}.$
\end{theorem}
\begin{proof}
First, we want to proof the case $q=2$ and later we extend the proof for $2<q<\infty$ by using a suitable embedding. Our main tool will be the following dispersive inequality (see \cite{Kara} pg. 114.)
\begin{equation}
\Vert u(t,x) \Vert_{L^{p'}_x(\mathbb{R}^n)}\leq C|t|^{-n|\frac{1}{p}-\frac{1}{2}|}\Vert f \Vert_{L^p(\mathbb{R}^n)},\,\,1<p<2.
\end{equation}
As consequence we have
\begin{equation}\label{estimatekool}
\Vert u(t,x) \Vert_{L^2_t([0,2\pi], L^{p'}_x(\mathbb{R}^n))    }\leq C\Vert |\,\cdot\,|^{-n|\frac{1}{p}-\frac{1}{2}|}\Vert_{L^2[0,2\pi]} \Vert f \Vert_{L^p(\mathbb{R}^n)},\,\,1<p<2.
\end{equation} We need $|\frac{1}{p}-\frac{1}{2}|<\frac{1}{2n}$ in order that $\Vert |\,\cdot\,|^{-n|\frac{1}{p}-\frac{1}{2}|}\Vert_{L^2[0,2\pi]}<\infty.$ Because $p'\geq 2$ we can use Minkowski integral inequality in order to obtain
\begin{equation}\label{LpT}
\Vert f \Vert_{\mathpzc{F}^0_{p',2}}=\Vert u(t,x)\Vert_{L^{p'}_x(\mathbb{R}^n, L^2_t([0,2\pi]))} \leq \Vert u(t,x) \Vert_{L^2_t([0,2\pi], L^{p'}_x(\mathbb{R}^n))    }\lesssim \Vert f \Vert_{L^p(\mathbb{R}^n)}.\, 
\end{equation} In fact we have,
\begin{align*}\Vert u(t,x)\Vert_{L^{p'}_x(\mathbb{R}^n, L^2_t([0,2\pi]))} &:=\left(\int\limits_{\mathbb{R}^n}\left( \int\limits_{0}^{2\pi} |u(t,x)|^2dt \right)^{\frac{p'}{2}}    dx\right)^{\frac{2}{p'}\cdot\frac{1}{2} }\\
&\leq \left( \int\limits_{0}^{2\pi} \left(\int\limits_{\mathbb{R}^n}  |u(t,x)|^{p'}dx \right)^{\frac{p'}{2}}dt    \right)^{\frac{1}{2} }=:\Vert u(t,x) \Vert_{L^2_t([0,2\pi], L^{p'}_x(\mathbb{R}^n))    }.
\end{align*}
Now \eqref{LpT} can be obtained from \eqref{estimatekool} for $1<p<2$ and $|\frac{1}{p}-\frac{1}{2}|<\frac{1}{2n}.$ The estimate \eqref{LpT}  proves the theorem for $q=2.$ The result for $2<q<\infty$ now follows, as in the proof of Theorem \ref{T1}, by using  the Wainger Sobolev embedding Theorem as in \eqref{31} together with \eqref{finalembedding}:
\begin{align*}
\Vert u(t,x) \Vert_{L^{p'}_x[\mathbb{R}^n,\,L^q_t[0,2\pi]]}\leq C &\Vert T'(H^{s_q}f) \Vert_{L^{p'}(\mathbb{R}^n)}\leq C_{p'}\Vert H^{s_q}f\Vert_{\mathpzc{F}^0_{p',2}(\mathbb{R}^n)}\\ 
&= C_{p'}\Vert f\Vert_{\mathpzc{F}^{s_q}_{p',2}(\mathbb{R}^n)} \leq C\Vert f\Vert_{W^{2s_q,p,H}(\mathbb{R}^n)}.
\end{align*} 
So, we end the  proof of the first announcement. Now, in order to proof \eqref{lplq}  we observe that 
\begin{equation}
\Vert u(t,x) \Vert_{L^{p'}_x(\mathbb{R}^n)}\leq C|t|^{-n|\frac{1}{p}-\frac{1}{2}|}\Vert f \Vert_{L^p(\mathbb{R}^n)},\,\,1<p<2,
\end{equation}
implies 
\begin{equation}
\Vert u(t,x) \Vert_{L_t^q[[0,2\pi],L^{p'}_x(\mathbb{R}^n)]}\leq C\cdot I_{p,n,q}\Vert f \Vert_{L^p(\mathbb{R}^n)},\,\,1<p<2,
\end{equation}
where $$I_{p,n,q}=(\int_{0}^{2\pi}|t|^{-nq|\frac{1}{p}-\frac{1}{2}|})^{\frac{1}{q}}<\infty $$
for $|1/2-1/p|<1/nq.$ Since, $q\leq p',$  by using the Minkowski inequality we have 
\begin{equation}
\Vert u(t,x) \Vert_{L^{p'}_x[\mathbb{R}^n,\,L^q_t[0,2\pi]]} \leq \Vert u(t,x) \Vert_{L_t^q[[0,2\pi],L^{p'}_x(\mathbb{R}^n)]}
\end{equation} and consequently
$$\Vert u(t,x) \Vert_{L^{p'}_x[\mathbb{R}^n,\,L^q_t[0,2\pi]]}\leq C\Vert f\Vert_{L^p}. $$
\end{proof}

\begin{theorem}\label{StrEst'}
Let us assume, for some $s,$ that $f\in \mathpzc{F}^s_{p,2}(\mathbb{R}^n)$ is a real function and $u(\cdot,t)=e^{-itH}f(\cdot).$ Let us assume $2\leq p\leq q<\infty $ and $\frac{2}{q}= n(\frac{1}{2}-\frac{1}{p}).$ Then, the following estimate
\begin{equation}\label{SECO}
\Vert e^{it\Delta}f\Vert_{L^q[(0,\infty), L_x^p(\mathbb{R}^n)]}\asymp \Vert u(t,x)\Vert_{L^q_{t}([0,2\pi]\,, L^p_{x}(\mathbb{R}^n))}\leq C \Vert f\Vert_{\mathpzc{F}^s_{p,2}(\mathbb{R}^n)},\,\, s\geq {s_q},
\end{equation}
holds true. Consequently we have
\begin{equation}\label{consequence}
\Vert e^{it\Delta}f\Vert_{L^q[(0,\infty), L_x^{p'}(\mathbb{R}^d)]}\asymp\Vert u(t,x)\Vert_{L^q_{t}([0,2\pi]\,, L^{p'}_{x}(\mathbb{R}^n))}\leq C \Vert f\Vert_{W^{2s,p,H}(\mathbb{R}^n)},\,\, s\geq {s_q},
\end{equation}
for $|\frac{1}{p}-\frac{1}{2}|<\frac{1}{2n},$ $1<p<2,$ $n>2$ and $\frac{2}{q}= n(\frac{1}{p}-\frac{1}{2}).$
Moreover, for $2\leq q\leq p<\infty$ and $\frac{2}{q}= n(\frac{1}{2}-\frac{1}{p})$ we have
\begin{equation}
\Vert f\Vert_{\mathpzc{F}^{0}_{p,2}(\mathbb{R}^n)}\leq C\Vert e^{it\Delta}f\Vert_{L^q[(0,\infty), L_x^p(\mathbb{R}^n)]},C\Vert u(t,x)\Vert_{L^q_{t}([0,2\pi]\,, L^p_{x}(\mathbb{R}^n))}.
\end{equation}
\end{theorem}
\begin{proof}
From the Minkowski integral inequality applied to $L^{\frac{q}{p}}$, we deduce the inequality
\begin{equation}
\Vert u(t,x)\Vert_{L^q_{t}([0,2\pi]\,, L^p_{x}(\mathbb{R}^n))}\leq \Vert u(t,x) \Vert_{L^p_x[\mathbb{R}^n,\,L^q_t[0,2\pi]]}.
\end{equation}
In fact, 
\begin{align*}
\Vert u(t,x)\Vert_{L^q_{t}([0,2\pi]\,, L^p_{x}(\mathbb{R}^n))} &:=\left(\int\limits_{0}^{2\pi}\left( \int\limits_{\mathbb{R}^n}|u(t,x)|^pdx \right)^{\frac{q}{p}}    dt\right)^{\frac{p}{q}\cdot\frac{1}{p} }\\
&\leq \left( \int\limits_{\mathbb{R}^n}\left(\int\limits_{0}^{2\pi}|u(t,x)|^qdt \right)^{\frac{p}{q}}dx    \right)^{\frac{1}{p} }=:\Vert u(t,x) \Vert_{L^p_x[\mathbb{R}^n,\,L^q_t[0,2\pi]]}.
\end{align*}

Now, we only need to apply Lemma \ref{T1} and the equivalence given by \eqref{SjToFormula}. The estimate \eqref{consequence} is consequence of \eqref{finalembedding} and \eqref{SECO} applied to $p'$ instead of $p$. On the other hand, for $2\leq q\leq p<\infty,$ by using the Minkowski integral inequality on $L^{\frac{p}{q}}$ we have 
\begin{equation}
\Vert f\Vert_{\mathpzc{F}^{0}_{p,2}(\mathbb{R}^n)}= \Vert u(t,x) \Vert_{L^p_x[\mathbb{R}^n,\,L^2_t[0,2\pi]]} \lesssim \Vert u(t,x) \Vert_{L^p_x[\mathbb{R}^n,\,L^q_t[0,2\pi]]}\leq \Vert u(t,x)\Vert_{L^q_{t}([0,2\pi]\,, L^p_{x}(\mathbb{R}^n))}.
\end{equation} So, by using  the equivalence expressed in  \eqref{SjToFormula} we obtain
$$\Vert f\Vert_{\mathpzc{F}^{0}_{p,2}(\mathbb{R}^n)}\leq C\Vert e^{it\Delta}f\Vert_{L^q[(0,\infty), L_x^p(\mathbb{R}^n)]}\asymp C\Vert u(t,x)\Vert_{L^q_{t}([0,2\pi]\,, L^p_{x}(\mathbb{R}^n))}$$ which end the proof of the theorem.
\end{proof}

\begin{corollary}\label{improv}
Let  $1<q\leq p<\infty$ and $\frac{1}{q}= \frac{n}{2}(\frac{1}{2}-\frac{1}{p}).$  Then,
\begin{equation}\label{impro'}
 \Vert u(t,x)\Vert_{L^p_{x}(\mathbb{R}^n\,, L^q_{t}[0,2\pi])}\leq C_s\Vert f\Vert_{{L}^2(\mathbb{R}^n)},
\end{equation}
holds true provided that, $n=1$ and $2\leq p\leq \infty,$ $n=2$ and $2\leq p<\infty$ and $2\leq p<\frac{2n}{n-2}$ for $n\geq 3.$ 
\end{corollary}
\begin{proof}
Same as in Theorem \ref{StrEst'}, by using the Minkowski integral inequality on $L^{\frac{p}{q}},$ for $1<q\leq p<\infty$, we have the inequality
\begin{equation}
\Vert u(t,x) \Vert_{L^p_x[\mathbb{R}^n,\,L^q_t[0,2\pi]]}\leq \Vert u(t,x)\Vert_{L^q_{t}([0,2\pi]\,, L^p_{x}(\mathbb{R}^n))} .
\end{equation}
Finally \eqref{impro'} now follows by using \eqref{KT'} and the equivalence \eqref{SjToFormula}.

\end{proof}

\end{document}